\documentclass[12pt,draft]{article} 
\usepackage{amsmath, amssymb, amsthm, amsfonts, indentfirst}
\usepackage{enumerate, color, bm, graphicx, here}
\makeatletter
\def\@cite#1#2{[{{\bfseries #1}\if@tempswa , #2\fi}]}
\renewcommand{\section}{%
\@startsection{section}{1}{\z@}
{0.5truecm plus -1ex minus -.2ex}%
{1.0ex plus .2ex}{\bfseries\large}}
\def\@seccntformat#1{\csname the#1\endcsname.\ }
\makeatother

\numberwithin{equation}{section} 
\newtheorem{thm}{Theorem}[section]
\newtheorem{corollary}[thm]{Corollary}
\newtheorem{lem}[thm]{Lemma}

\theoremstyle{definition}

\newtheorem{remark}{Remark}[section]

\newcommand{\ep}{\varepsilon}
\newcommand{\pa}{\partial}

\newcommand{\ol}{\overline}
\newcommand{\wt}{\widetilde}
\newcommand{\tmax}{T_{\max}}

\newcommand{\lp}[2]{\| #2 \|_{L^{#1}(\Omega)}}
\newcommand{\ctime}{T_0}
\newcommand{\cKlam}{L}
\newcommand{\Ko}{K_1}
\newcommand{\Ktw}{K_2} 
\newcommand{\Kth}{K_3}
\newcommand{\Kf}{K_4}
\newcommand{\cChat}{C_2}
\newcommand{\cCwt}{C_1}
\newcommand{\cmu}{\xi}
\newcommand{\mass}{m_0}
\newcommand{\wmax}{b}
\newcommand{\umax}{\ol{M}}

\begin{document}
\footnote[0]
    {2020{\it Mathematics Subject Classification}\/. 
    Primary: 35B40; 
    Secondary: 35K45, 35Q92, 92C17.
    }
\begin{center}
\Large{{\bf 
Boundedness and asymptotic stability
in a model \\
for tuberculosis granuloma formation}}
\end{center}
\vspace{5pt}
\begin{center}
    Masaaki Mizukami
    \\
               \vspace{8pt}
    Department of Mathematics, Faculty of Education, Kyoto University of Education\\
    1, Fujinomori, Fukakusa, Fushimi-ku, Kyoto, 612-8522, Japan\\
               \vspace{15pt}
    Yuya Tanaka
    \footnote{Corresponding author.}\\
               \vspace{8pt}
    Department of Mathematical Sciences,
    Kwansei Gakuin University\\
    1, Uegahara, Gakuen, Sanda, Hyogo, 669-1330, Japan\\ 
        \footnote[0]{{E-mail addresses}:
    {\tt masaaki.mizukami.math@gmail.com},
    {\tt yuya.tns.6308@gmail.com}}
               \vspace{15pt}
\end{center}
\begin{center}    
    \small \today
\end{center}

\vspace{2pt}
\newenvironment{summary}
{\vspace{.5\baselineskip}\begin{list}{}{%
     \setlength{\baselineskip}{0.85\baselineskip}
     \setlength{\topsep}{0pt}
     \setlength{\leftmargin}{12mm}
     \setlength{\rightmargin}{12mm}
     \setlength{\listparindent}{0mm}
     \setlength{\itemindent}{\listparindent}
     \setlength{\parsep}{0pt}
     \item\relax}}{\end{list}\vspace{.5\baselineskip}}
\begin{summary}
{\footnotesize {\bf Abstract.} 
This paper deals with a problem which describes tuberculosis granuloma formation
  \begin{align*}
  \begin{cases}
    u_t = \Delta u - \nabla \cdot (u \nabla v) 
            - uv - u + \beta,
    &x \in \Omega,\ t>0, \\
    v_t = \Delta v + v -uv +  \mu w,
    &x \in \Omega,\ t>0, \\
    w_t = \Delta w + uv - wz - w,
    &x \in \Omega,\ t>0, \\
    z_t = \Delta z - \nabla \cdot (z \nabla w)
            + f(w)z -z,
    &x \in \Omega,\ t>0
  \end{cases}
  \end{align*}
under homogeneous Neumann boundary conditions and initial conditions, where $\Omega \subset \mathbb{R}^n$ ($n\ge 2$)
is a smooth bounded domain, 
$\beta,\mu>0$ and $f$ is some function, and shows that, 
if the reproduction number $R_0 := \frac{\mu \beta + 1}{\beta}$ satisfies $R_0<1$
and
initial data are small in some sense, 
then the solution $(u,v,w,z)$ of the problem exists globally and convergences to $(\beta,0,0,0)$ exponentially. 
}
\end{summary}
\vspace{10pt}

\newpage

%
%
\section{Introduction}
\noindent\textbf{Bachground.}  
Chemotaxis is a property such that cells or species react on some chemical substances, and move toward higher gradient of that substance. 
Here it is known that chemotaxis plays an important role in not only biology but also in medicine and engineering. 
One of the well-studied problems which describe chemotactic aggregation phenomena is a Keller--Segel system 
\begin{align*}
u_t = \Delta u - \nabla \cdot (u\nabla v), \quad v_t = \Delta v-v +u, \qquad x\in \Omega, \ t >0,   
\end{align*}
under homogeneous Nuemann boundary conditions and initial conditions. 
In this system the term $ - \nabla \cdot (u\nabla v) $ represents chemotactic aggregation effect, and causes difficulties in analyzing this system; for more details, see surveys \cite{B-B-T-W,LW_Survey,Arumugam-Tyagi}. 

The modeling and analysis of the Keller--Segel systems are extended to variations which describe not only biological phenomena (e.g., two-species chemotaxis-competition problems \cite{Tello_Winkler_2012,MTY_BU,PanMuTao_2023}, aggregation of cells in a drop of water \cite{Tuvaletal,W-2012,Winkler_2016}) but also medical phenomena (e.g., Alzheimer's disease \cite{AlzheimerModeling}, the tumor invasion \cite{FujieItoYokota,FIWY,Ishida-Yokota_2023}, 
the tumor angiogenesis \cite{OrmeChaplain,TW2021_angiogenesis,Chiyo-Mizukami}).
Recently, as to the tuberculosis granuloma formation, 
Feng \cite{Feng2024} developed mathematical models of the evolution of granuloma by the Keller--Segel system, 
and proposed the following problem
  \begin{align}\label{OriginalProb}
  \begin{cases}
    u_t = D_u \Delta u - \chi_u \nabla \cdot (u \nabla v) 
            -\gamma_u uv - \delta_u u + \beta_u,
    &x \in \Omega,\ t>0, \\
    v_t = D_v\Delta v + \rho_v v - \gamma_v uv +  \mu_v w,
    &x \in \Omega,\ t>0, \\
    w_t = D_w \Delta w + \gamma_w uv - \alpha_w wz - \mu_w w,
    &x \in \Omega,\ t>0, \\
    z_t = D_z 
    \Delta z - \chi_z \nabla \cdot (z \nabla w)
            + \alpha_z f(w)z - \delta_z z,
    &x \in \Omega,\ t>0
  \end{cases}
  \end{align}
under homogeneous Neumann boundary conditions and initial conditions, 
where 
\[D_u,D_v,D_w,D_z,\chi_u,\chi_z,\gamma_u,\gamma_v,\gamma_w,\mu_v,\mu_w,\alpha_w,\alpha_z,\delta_u,\delta_z,\rho_v,\beta_u>0\] 
and 
\[ f(w) = \frac{w}{1+w} \quad \mbox{or} \quad f(w) = w. \]
Here, unknown functions $u,v,w,z$ describe the density of healthy macrophages, bacteria, infected macrophages, CD4 T cells, respectively 
(For more detail explanations of the model and history of modeling, see \cite{Feng2024,FLM3}).

\medskip
\noindent \textbf{Previous works and main result.} 
Feng \cite{Feng2024} first dealt with an ODE problem of \eqref{OriginalProb} with $\gamma_u=\gamma_w =: \gamma_i$ and $f(w)= \frac{w}{1+w}$, and showed that the reproduction number 
\[
  R_0 := \frac{s\beta_u \gamma_i + \rho_v\delta_u}{\beta_u \gamma_v} \qquad \left(s=\frac{\mu_v}{\mu_w} \right)
\]
was important for analysis of boundary equilibria: 
The infection-free equilibrium $E_0 := (\frac{\beta_u}{\delta_u},0,0,0)$ is locally asymptotically stable when $R_0 < 1$; existence and complicated analysis of an immune-free equilibrium $E_1 := (u^\ast,v^\ast,w^\ast,0)$ were shown when $R_0 <1$. 
Feng also analyzed other kinds of equilibrium of the ODE problem, and numerical simulation not only the ODE problem but also the PDE problem in 1-/2-dimensional settings.

About the mathematical analysis of the original problem \eqref{OriginalProb}, 
global existence of classical/weak solutions was established in 2-/3-dimensional settings, respectively in \cite{FLM3}. 
Nevertheless, 
at least two problems are left: 
{\it Boundedness of the solutions};  
{\it global existence of classical solutions in $n$-dimensional setting $(n\ge 3)$}. 
The purpose of this paper is to give some answer to
these problems; in this paper for the simplicity we let almost constants in \eqref{OriginalProb} except essential constants ($\beta_u=\beta$, $\mu_v = \mu$) be equal to 1, and consider the following problem:
  \begin{align}\label{P}
  \begin{cases}
    u_t = \Delta u - \nabla \cdot (u \nabla v) 
            - uv - u + \beta,
    &x \in \Omega,\ t>0, \\
    v_t = \Delta v + v -uv +  \mu w,
    &x \in \Omega,\ t>0, \\
    w_t = \Delta w + uv - wz - w,
    &x \in \Omega,\ t>0, \\
    z_t = \Delta z - \nabla \cdot (z \nabla w)
            + f(w)z -z,
    &x \in \Omega,\ t>0, \\
    \nabla u \cdot \nu 
    = \nabla v \cdot \nu 
    = \nabla w \cdot \nu
    = \nabla z \cdot \nu
    = 0,
    &x \in \partial \Omega,\ t>0, \\
    u(\cdot,0) = u_0, 
    \quad
    v(\cdot,0) = v_0, 
    \quad
    w(\cdot,0) = w_0,
    \quad
    z(\cdot,0) = z_0,
    &x \in \Omega, 
  \end{cases}
  \end{align}
where $\Omega \subset \mathbb{R}^n$ ($n \ge 2$) 
is a bounded domain with smooth boundary $\partial\Omega$, 
$\beta, \mu>0$, 
$\nu$ is the outward normal vector to $\pa\Omega$, 
and $u_0, v_0, w_0, z_0 \geq 0$ in $\Omega$.
Here we note that, since $\mu_v=\mu$ and  $\beta_u = \beta$ as well as $\mu_w = \gamma_v = \gamma_w = \rho_v = \delta_u=1$ , the reproduction number of \eqref{P} is 
  \[
    R_0 = \frac{\mu \beta + 1}{\beta}.
  \]

The main result reads as follows. 
In this result we showed not only global existence and boundedness of classical solutions 
to \eqref{P} but also its asymptotic behavior.
More precisely, 
the infection-free equilibrium $E_0=(\beta,0,0,0)$ is asymptotically stable when $R_0<1$ and initial data $v_0$ and $w_0$ are small in some sense.

\begin{thm}\label{thm}
Let $\Omega \subset \mathbb{R}^n$ $(n\ge2)$ be a smooth, bounded domain
and let $\beta, \mu>0$ and $q>n$.
Suppose that 
  \begin{align}\label{f}
    f \in C^1(\mathbb{R}),
    \qquad 
    0 \le f(s) \le s
    \quad\mbox{for all $s\in(0,\infty)$}
  \end{align}
and 
  \begin{align}\label{rep}
    R_0 = \frac{\mu \beta + 1}{\beta} < 1.
  \end{align}
Then for all $\alpha>0$ and 
$\xi\in\big(\mu,\frac{\beta-1}{\beta}\big)$ 
there exist $\gamma, \ep_0, C>0$
with the following property\/{\rm :}
Whenever initial data satisfies that 
  \begin{align}\label{initial}
    (u_0,v_0,w_0,z_0) \in C^0(\ol{\Omega}) \times W^{1,q}(\Omega) \times W^{1,q}(\Omega) \times C^0(\ol{\Omega})
    \quad\mbox{is nonnegative}
  \end{align}
and 
  \[
    \| u_0 - \beta \|_{L^\infty(\Omega)}
    =\alpha,
    \qquad
    \|v_0 + \xi w_0\|_{L^\infty(\Omega)} 
    \le \ep_0
    \quad\mbox{and}\quad
    \| \nabla v_0 \|_{L^q(\Omega)}
    \le 
    \sqrt{\ep_0},
  \]
then there exists a unique global classical solution of \eqref{P} which fulfills that
  \[
    \| u(\cdot,t)-\beta \|_{L^\infty(\Omega)}
    + \| v(\cdot,t) \|_{W^{1,q}(\Omega)}
    + \| w(\cdot,t) \|_{W^{1,q}(\Omega)}
    + \| z(\cdot,t) \|_{L^\infty(\Omega)}
    \le C e^{-\gamma t}
  \]
for all $t\in(0,\infty)$.
\end{thm}

\begin{remark}
``$\beta, \mu>0$ and \eqref{rep}'' is equivalent to 
``$\beta>1$ and $0<\mu < \frac{\beta-1}{\beta}$''.
\end{remark}

\noindent\textbf{Main ideas and plan of the paper.}
To determine a strategy we first recall what the previous work \cite{FLM3} did and what become difficulties in attaining our purpose. 
The previous work focused on the fact that 
the $(u,v)$- and $(z,w)$-subsystem in \eqref{P} share some characteristics with 
the chemotaxis-consumption system 
\begin{align*}
    n_t = \Delta n - \nabla \cdot (n\nabla c), \quad c_t = \Delta c -nc
\end{align*}
under homogeneous
Neumann boundary conditions and initial conditions, 
and showed that a modified energy functional enabled us to obtain existence of global classical/weak solutions to \eqref{OriginalProb} in 2-/3-dimensional setting. 
Here, one of the reason why the previous work could not attain boundedness of solutions is in the second equation in \eqref{P} 
\[
  v_t = \Delta v + v -uv +  \mu w,
\]
especially $+v$; since we could not control this term (even though there is the term $- uv$), 
we only obtained ``time-growth'' type estimates like 
$\int_\Omega v \lesssim e^{t}$, which could not yield global boundedness of solutions. 
Thus the start point of this study is come from a question: Could we control $+v$ by using $-uv$? 
Then, to give some answer to this question we try to apply the following idea: If the solution of \eqref{P} converges to the infection-free equilibrium $(\beta,0,0,0)$, especially if $u(\cdot,t)\to \beta$ ($t\to \infty$), then $+ v -uv  \to  - (\beta - 1) v$ ($t\to \infty$), 
which with $\beta >1$ might derive global estimates for the solution. 
To realize this idea mathematically, 
after collecting elementary results including existence of the local solution to \eqref{P} in Section \ref{Sec;Preli}, 
we introduce  
\[
    T := 
    \sup 
    \Big\{ \tau \in (0,\tmax) \ ;\ 
      \| u(\cdot,t) - \beta \|_{L^\infty(\Omega)} 
      \le g(t)  
      \quad \mbox{for all $t \in [0,\tau)$}
    \Big\} >0, 
\]
where $g$ is some function satisfying $g(t)\to 0$ $(t\to \infty)$.  
Section \ref{Sec;BddOfUVW} is devoted to showing $T=\tmax$; in Section \ref{SubSec;EstiVW} 
by applying the above idea
we first derive the estimate
\[
  \| v(\cdot,t)\|_{W^{1,q}(\Omega)} + \lp{\infty}{w(\cdot,t)} \le h_1(t) \quad \mbox{for all} \ t\in (0,T)
\]
with some function $h_1$ fulfilling 
$h_1(t) \to 0$ ($t\to \infty$), and then in Section \ref{Subsec;estiU} we see 
\[
  \lp{\infty}{u(\cdot,T) - \beta} < g(T) 
\]
under some smallness condition for the initial data, which entails that $T=\tmax$. 
Moreover, throughout Section \ref{Sec;BddOfZ} we prove
existence of a function $h_2$ satisfying 
\[
  \lp{\infty}{z(\cdot,t)}+ \lp{q}{\nabla w(\cdot,t)} \le h_2(t)
\]
and $h_2(t)\to 0$ ($t\to \infty$). These estimates derive Theorem \ref{thm}.  

%
\section{Preliminary}\label{Sec;Preli}

In this section we collect elementary results. 
We first state local existence of classical solutions 
to \eqref{P}, which is given by \cite[Lemma 3.1]{FLM3}. 

\begin{lem}\label{local}
Let $\Omega \subset \mathbb{R}^n$ $(n\ge2)$ 
be a smooth, bounded domain and let $\beta,\mu>0$.
Suppose that \eqref{f}
and \eqref{initial} hold for some $q>n$.
Then there exist $\tmax \in(0,\infty]$ and a uniquely determined solution
  \[
    (u,v,w,z)\in C^0([0,\tmax);C^0(\ol{\Omega}) \times W^{1,q}(\Omega) \times W^{1,q}(\Omega) \times C^0(\ol{\Omega}))
    \cap \big( C^{2,1}(\ol{\Omega}\times(0,\tmax)) \big)^4
  \]
of \eqref{P} which is nonnegative and has the property that 
  \begin{align}\label{criterion}
    \mbox{if $\tmax<\infty$}, \quad\mbox{then}\quad
    \limsup_{t\nearrow\tmax}\left( \|u(\cdot,t)\|_{L^\infty(\Omega)} + \|z(\cdot,t)\|_{L^\infty(\Omega)} \right) = \infty.
  \end{align}
\end{lem}

The previous work \cite{FLM3} has already shown that  the unique classical solution exists globally in the case that $n=2$. 
\begin{corollary}[{\cite[Theorem 1.1]{FLM3}}]\label{cor2d}
  In the case that $n=2$ we have $\tmax = \infty$.  
\end{corollary}

We next give a basic property of solutions to \eqref{P}.

\begin{lem}\label{mass}
Let $\Omega \subset \mathbb{R}^n$ $(n\ge2)$
be a smooth, bounded domain and let $\beta,\mu>0$.
Suppose \eqref{f} and \eqref{initial} for some $q>n$. 
Then 
  \[
    \| u+w+z \|_{L^1(\Omega)}
    \le \| u_0+w_0+z_0 \|_{L^1(\Omega)} + \beta |\Omega|
    =:\mass
    \quad\mbox{on $[0,\tmax)$}.
  \]
\end{lem}

\begin{proof}
From the first, third and fourth equations in \eqref{P}
and the relation \eqref{f},
we have
  \[
    \frac{d}{dt} \int_\Omega (u+w+z)
    \le - \int_\Omega (u+w+z) + \beta |\Omega|
    \quad\mbox{on $(0,\tmax)$},
  \]
which leads to the conclusion of this lemma.
\end{proof}

Next we recall $L^p$-$L^q$ estimates for the 
Neumann heat semigroup (see \cite[Lemma 1.3]{W_2010} and \cite[Lemma 2.1]{CL_2016}). 
\begin{lem}\label{semi}
Let $(e^{t\Delta})_{t\ge0}$ be the Neumann heat semigroup 
in $\Omega$, and let $\lambda>0$ denote the 
first nonzero eigenvalue of $-\Delta$ in $\Omega$
under Neumann boundary condition.
Then there are $\Ko,\Ktw,\Kth,\Kf >0$ which only depend on $\Omega$ and which have the following properties\/{\rm :} 
\begin{enumerate}
  \item If $1\le q\le p \le \infty$, then 
  \[  
    \lp{p}{e^{t\Delta} \phi} \le 
        \Ko (1+ t^{-\frac n2 (\frac 1q - \frac 1p)}) e^{-\lambda t}\lp{q}{\phi} 
    \quad 
    \mbox{for all} \ t>0 
  \]
  holds for all $\phi\in L^q (\Omega)$ with $\int_\Omega \phi = 0$. 
  \item  If $1\le q\le p \le \infty$, then 
  \[
    \lp{p}{\nabla e^{t\Delta} \phi} \le 
      \Ktw (1+ t^{-\frac 12 -\frac n2 (\frac 1q - \frac 1p)}) e^{-\lambda t}\lp{q}{\phi} 
    \quad 
    \mbox{for all} \ t>0 
  \]
  holds for all $\phi\in L^q (\Omega)$. 
  \item If $2\le q\le p \le \infty$, then 
  \[
    \lp{p}{\nabla e^{t\Delta} \phi} \le 
      \Kth (1+ t^{-\frac n2 (\frac 1q - \frac 1p)}) e^{-\lambda t}\lp{q}{\nabla \phi} 
    \quad 
    \mbox{for all} \ t>0 
  \]
  holds for all $\phi\in W^{1,q} (\Omega)$. 
  \item If $1 <  q \le p \le \infty$, then 
  \[
    \lp{p}{e^{t\Delta} \nabla \cdot \varphi} \le 
      \Kf (1+ t^{-\frac 12 -\frac n2 (\frac 1q - \frac 1p)}) e^{-\lambda t}\lp{q}{\varphi} 
    \quad 
    \mbox{for all} \ t>0 
  \]
  holds for all $\varphi\in (W^{1,q} (\Omega))^n$.
\end{enumerate}
\end{lem}

\section{Decay estimates for {\boldmath $v,\nabla v,w$} and {\boldmath $u-\beta$}}\label{Sec;BddOfUVW}

Throughout the sequel, we fix 
$\beta ,\mu>0$ with \eqref{rep} (which is equivalent to 
$\beta>1$ and $0<\mu < \frac{\beta-1}{\beta}$), and let $\eta >0$ and $q >n$ as well as $\alpha>0$.
Moreover, without further explicit mentioning, we assume \eqref{f} and \eqref{initial} with $\alpha = \| u_0 - \beta \|_{L^\infty(\Omega)}$, and $(u,v,w,z)$ denotes the corresponding local classical solution in \eqref{P} given by Lemma \ref{local}.
We let 
  \begin{align}\label{dxg}
    \xi \in \left( \mu, \frac{\beta-1}\beta \right),
  \quad 
    \delta < (1-\xi)\beta-1
  \quad \mbox{and}\quad 
    \gamma < \min\left\{ \delta, 1 - \frac\mu\xi, \lambda \right\}
  \end{align}
%
and put
  \begin{align}\label{T}
    T := 
    \sup 
    \Big\{ \tau \in (0,\tmax) \ ;\ 
      \| u(\cdot,t) - \beta \|_{L^\infty(\Omega)} 
      \le g(t)  
      \quad \mbox{for all $t \in [0,\tau)$}
    \Big\}, 
  \end{align}
where 
\[
g(t) := \alpha e^{-t} + \cKlam \| \nabla v_0 \|_{L^q(\Omega)} e^{-\lambda t} + \eta e^{-\gamma t}
\]
and
\begin{align}\label{cKlam}
\cKlam = \cKlam (\alpha,\beta,\eta,q,\Omega) := K_3 K_4 (\alpha + \beta + \eta) \int^\infty_0 (1+\sigma^{-\frac12-\frac{n}{2q}})e^{-\sigma} \,d\sigma < \infty
\end{align}
as well as $\lambda, K_3, K_4>0$ are constants given in Lemma \ref{semi}. Here, since $\| u(\cdot,0) - \beta \|_{L^\infty(\Omega)} = \alpha  < \alpha + \cKlam \| \nabla v_0 \|_{L^q(\Omega)} + \eta$ holds, we note that $T$ is well-defined and $T\in (0,\tmax]$. 
In this section, we shall show $T=\tmax$ 
to obtain boundedness and decay estimates for $v, \nabla v, w$ and $u-\beta$.  

\subsection{Estimates for {\boldmath $v,\nabla v,w$} on {\boldmath $[0,T)$}}\label{SubSec;EstiVW}

In this subsection we will derive decay estimates for $v,\nabla v,w$ on $[0,T)$. 
We first show estimates for $v,w$, which is actually an essential point in the proof, according to the following idea: 
for $\xi >0$ the function $v+ \xi w$ satisfies that 
\[
  (v+\xi w)_t = \Delta (v+\xi w) - ((1-\xi) u - 1) v - (\xi - \mu)w - \xi wz,
\]
which with the conjecture $u(\cdot,t)\to \beta$ ($t\to \infty$) might entail that, if 
both $(1-\xi) \beta -1 >0$ and $\xi - \mu >0$ hold (which is equivalent to $\xi \in (\mu,\frac{\beta -1}{\beta})$), then we can obtain uniform-in-time boundedness of $v +\xi w$. 

\begin{lem}\label{full:Linf:v}
Let $\xi, \delta, \gamma>0$ satisfy \eqref{dxg}.
Moreover, assume that 
$\|\nabla v_0\|_{L^q(\Omega)}\le 1$.
Then 
  \begin{align}\label{esti:vw}
    \| v(\cdot,t) + \cmu w(\cdot,t) \|_{L^\infty(\Omega)} 
    \le \| v_0 + \cmu w_0 \|_{L^\infty(\Omega)} 
         e^{(1+\gamma)\ctime}  
         e^{-\gamma t}
    \quad\mbox{for all $t\in[0,T)$},
  \end{align}
where
\[
  \ctime := 
  \begin{cases}
      0 
    & \mbox{if} \ \alpha + \cKlam + \eta \le \beta - \frac{1+\delta}{1-\xi},
    \\
      \frac1{\gamma} 
      \log \frac{\alpha + \cKlam + \eta}{\beta - \frac{1+\delta}{1-\xi}} 
    &\mbox{if} \ \alpha + \cKlam + \eta > \beta -\frac{1+\delta}{1-\xi}.
  \end{cases}
\]
\end{lem}

\begin{remark}
We can actually choose 
$\xi,\delta$ and $\gamma$ satisfying \eqref{dxg}.
Indeed, the interval $(\mu, \frac{\beta-1}{\beta})$
is not empty
since $\beta,\mu>0$ with \eqref{rep} meas $\beta>1$ and $\mu<\frac{\beta-1}\beta$; 
also, the relations $\xi < \frac{\beta-1}{\beta}$ and $\xi>\mu$ 
are equivalent to $(1-\xi)\beta-1>0$
and $1-\frac\mu\xi>0$, respectively, 
which ensures existence of $\gamma$ and $\delta$ in \eqref{dxg}.
Moreover, we note that the second relation in \eqref{dxg} guarantees
that $\beta - \frac{1+\delta}{1-\xi}>0$.
\end{remark}

\begin{proof}
From the second and third equations in \eqref{P}
and the fact that $\mu < \xi <\frac{\beta-1}{\beta}< 1$, 
  \begin{align*}
    (v+\xi w)_t
    &=\Delta (v+\xi w) + (v + \xi w) - \xi w - (1-\xi) uv - (\xi- \mu) w - \xi wz
    \\
    &\le \Delta (v+\xi w) + (v+\xi w) 
    \quad\mbox{on $(0,\tmax)$}, 
  \end{align*}
which implies that 
  \begin{align}\label{short}
    \| v(\cdot,t)+\xi w(\cdot,t) \|_{L^\infty(\Omega)}
    \le \| v_0+\xi w_0 \|_{L^\infty(\Omega)} e^t
    \quad\mbox{for all $t\in[0,\tmax)$}.
  \end{align}
Because $e^t = e^{(1+\gamma)t-\gamma t} \le e^{(1+\gamma)\ctime} e^{-\gamma t}$ for all $t\in [0,T)\cap[0,\ctime)$, 
we have
  \[
    \| v(\cdot,t) + \cmu w(\cdot,t) \|_{L^\infty(\Omega)} 
    \le \| v_0 + \cmu w_0 \|_{L^\infty(\Omega)} 
         e^{(1+\gamma)\ctime}  
         e^{-\gamma t}
     \quad\mbox{for all $t\in [0,T)\cap[0,\ctime)$}.
  \]
If $\ctime\ge T$, then this inequality directly  derives \eqref{esti:vw}.
Therefore we next consider the case $\ctime<T$.
In light of the condition that 
$\|\nabla v_0\|_{L^q(\Omega)}\le1$ 
and the definition of $\ctime$ we have 
  \[
    \alpha e^{-\ctime} +  \cKlam  \| \nabla v_0 \|_{L^q(\Omega)} e^{-\lambda \ctime} + \eta e^{-\gamma \ctime}
    \le (\alpha  +  \cKlam  + \eta)e^{-\gamma \ctime}
    \le \beta - \frac{1+\delta}{1-\xi}.
  \]
Then we see from the definition of $T$ in \eqref{T} that 
$\| u - \beta \|_{L^\infty(\Omega)} \le \beta - \frac{1+\delta}{1-\xi}$ on $[\ctime,T)$, 
which ensures that
  \begin{align}\label{lower:u}
    \frac{1+\delta}{1-\xi} \le u
    \quad\mbox{on $[\ctime,T)$}.
  \end{align}
From the second and third equations in \eqref{P} 
and \eqref{dxg} as well as \eqref{lower:u} it follows that 
  \begin{align*}
    (v+\xi w)_t
    &\le \Delta (v+\xi w) + v - (1-\xi)uv - (\xi-\mu)w
    \\
    &\le \Delta (v+\xi w) - \delta v - (\xi-\mu) w
    \\
    &\le \Delta (v+\xi w) - \gamma (v + \xi w)
    \quad\mbox{on $[\ctime,T)$},
  \end{align*}
and hence, by virtue of semigroup estimates in Lemma \ref{semi} 
we observe that
  \begin{align}\label{long}
    \|v(\cdot,t)+\xi w(\cdot,t)\|_{L^\infty(\Omega)}
    \le \|v(\cdot,\ctime)+\xi w(\cdot,\ctime)\|_{L^\infty(\Omega)}
         e^{-\gamma(t-\ctime)}
    \quad\mbox{for all $t\in[\ctime,T)$}.
  \end{align}
Thus inserting \eqref{short} with $t=\ctime$ into 
\eqref{long} leads to \eqref{esti:vw} on $[\ctime,T)$. 
\end{proof}

Lemma \ref{full:Linf:v} enables us to see a decay estimate for $\nabla v$ on $[0,T)$. 

\begin{lem}\label{full:Linf:nabv}
Let $\xi, \delta, \gamma>0$ satisfy \eqref{dxg}. 
Moreover, assume that 
$\|\nabla v_0\|_{L^q(\Omega)}\le 1$.
Then 
  \[
    \| \nabla v(\cdot,t) \|_{L^q(\Omega)}
    \le \Kth \| \nabla v_0 \|_{L^q(\Omega)} e^{-\lambda t}
           + C_1 \| v_0 + \cmu w_0 \|_{L^\infty(\Omega)}  e^{-\gamma t} 
    \quad \mbox{for all $t\in[0,T)$},
  \]
where 
  \begin{align*}
    \cCwt & := \Ktw S_1 |\Omega|^{\frac 1q} 
              \left(M+\frac{\mu}{\xi}\right) 
              e^{(1+\gamma)\ctime}, 
    \quad  
    M := \alpha + \beta +\eta + \cKlam  \| \nabla v_0 \|_{L^q(\Omega)} -1,  
  \\
    S_1& := \int^\infty_0 (1+\sigma^{-\frac12})e^{-(\lambda-\gamma) \sigma} \,d\sigma.
  \end{align*}
\end{lem}

\begin{proof}
By virtue of the definition of $T$ in \eqref{T} 
and the fact that $\beta > 1$ we obtain
  \begin{align}\label{u-1}
    \| u(\cdot,t) - 1 \|_{L^\infty(\Omega)} 
    &\le \lp{\infty}{u(\cdot,t)-\beta} + \lp{\infty}{\beta -1}
\notag \\
    &\le \alpha 
      +\cKlam \| \nabla v_0 \|_{L^q(\Omega)} 
      + \eta 
      + \beta - 1 
    = M
    \quad \mbox{for all $t\in[0,T)$}.
  \end{align}
The second equation in \eqref{P} and Lemmas \ref{semi}, \ref{full:Linf:v} as well as \eqref{u-1}
entail that
  \begin{align*}
    &\| \nabla v(\cdot,t) \|_{L^q(\Omega)}
    \\
    &\le \| \nabla e^{t\Delta} v_0 \|_{L^q(\Omega)}
           + \int^t_0 \| \nabla e^{(t-s)\Delta} [ (1-u)v + \mu w] \|_{L^q(\Omega)} \,ds 
    \\
    &\le \Kth e^{-\lambda t} \| \nabla v_0 \|_{L^q(\Omega)}
           + \Ktw  
              \int^t_0 (1+(t-s)^{-\frac12})e^{-\lambda (t-s)} 
              \left( M \| v \|_{L^q(\Omega)} + \frac{\mu}{\xi} \| \xi w\|_{L^q(\Omega)} \right) \,ds
    \\
    &\le \Kth e^{-\lambda t} \| \nabla v_0 \|_{L^q(\Omega)}
    \\   &\quad\,     
           + \Ktw |\Omega|^{\frac 1q} 
              \left(M+\frac{\mu}{\xi}\right) 
              e^{(1+\gamma)\ctime}   
              \| v_0 + \cmu w_0 \|_{L^\infty(\Omega)} 
              \int^t_0 (1+(t-s)^{-\frac12})e^{-\lambda (t-s)} e^{-\gamma s} \,ds
  \end{align*}
for all $t\in[0,T)$. 
Here, substituting $\sigma = t-s$, we infer that 
  \begin{align*}
    \int^t_0 (1+(t-s)^{-\frac12})e^{-\lambda (t-s)} e^{-\gamma s} \,ds
    &= \int^t_0 (1+\sigma^{-\frac12})e^{-\lambda \sigma} e^{-\gamma (t-\sigma)} \,d\sigma \\
    &\le e^{-\gamma t} \int^t_0 (1+\sigma^{-\frac12})e^{-(\lambda-\gamma) \sigma} \,d\sigma \\
    &\le S_1 e^{-\gamma t},
  \end{align*}
where $S_1=\int^\infty_0 (1+\sigma^{-\frac12})e^{-(\lambda-\gamma) \sigma} \,d\sigma < \infty$. 
Hence, combining the above inequalities leads to
the estimate
  \begin{align*}
    \| \nabla v(\cdot,t) \|_{L^q(\Omega)}
    &\le \Kth  \| \nabla v_0 \|_{L^q(\Omega)} e^{-\lambda t}
    \\
    &\quad\,
           + \Ktw S_1 |\Omega|^{\frac 1q} 
              \left(M+\frac{\mu}{\xi}\right) 
              e^{(1+\gamma)\ctime}   
              \| v_0 + \cmu w_0 \|_{L^\infty(\Omega)}  e^{-\gamma t}
  \end{align*}
for all $t\in[0,T)$, 
which implies the conclusion.
\end{proof}

\subsection{Proof of the fact {\boldmath $T=\tmax$}}\label{Subsec;estiU}

In this subsection we will see that $T = \tmax$. 
We first re-obtain a decay estimate for $u-\beta$ by using Lemmas \ref{full:Linf:v} and \ref{full:Linf:nabv}. 

\begin{lem}\label{full:lem;esti;u-beta}
Let $\xi, \delta, \gamma>0$ satisfy \eqref{dxg}. 
Moreover, assume that 
$\|\nabla v_0\|_{L^q(\Omega)}\le 1$.
Then
  \[
    \| u(\cdot,t) - \beta \|_{L^\infty(\Omega)}
    \le \alpha e^{-t}
           + \Kth\Kf S_2 (M+1) \| \nabla v_0 \|_{L^q(\Omega)} 
              e^{-\lambda t} 
           + \cChat \|v_0 + \xi w_0 \|_{L^\infty(\Omega)} e^{-\gamma t}
 \]
for all $t\in[0,T)$, 
where 
  \begin{align*}
    \cChat 
    &:= \left[
            \Ktw \Kf S_1S_3 |\Omega|^{\frac 1q} 
            \left(M+\frac{\mu}{\xi}\right) 
            + \frac{1}{1-\gamma} 
          \right] 
          (M+1)e^{(1+\gamma)\ctime}, 
  \\
    S_2 &:=\int^\infty_0 (1+\sigma^{-\frac12-\frac{n}{2q}})e^{-\sigma} \,d\sigma, \quad 
    S_3:=\int^\infty_0 (1+\sigma^{-\frac12-\frac{n}{2q}})e^{-(\lambda+1-\gamma) \sigma} \,d\sigma. 
  \end{align*}
\end{lem}

\begin{proof}
From the first equation in \eqref{P} and semigroup estimates in Lemma \ref{semi}, 
we have
  \begin{align*}
    &\| u(\cdot,t)-\beta \|_{L^\infty(\Omega)}
    \\
    &\le \| e^{t(\Delta-1)}(u_0-\beta) \|_{L^\infty(\Omega)}
           + \int^t_0 
              \big\| e^{(t-s)(\Delta-1)}
                [ - \nabla \cdot (u \nabla v) - uv ] 
              \big\|_{L^\infty(\Omega)}\,ds \\
    &\le e^{-t} \| u_0-\beta \|_{L^\infty(\Omega)} 
           + \Kf \int^t_0 (1+(t-s)^{-\frac12-\frac{n}{2q}})e^{-(\lambda+1)(t-s)}
                     \|u \nabla v\|_{L^q(\Omega)}\,ds \\
    &\quad\,
           + \int^t_0 e^{-(t-s)} \| uv \|_{L^\infty(\Omega)}\,ds
    \quad\mbox{for all $t\in[0,T)$}.
  \end{align*}
As in \eqref{u-1} we know that $\|u-1\|_{L^\infty(\Omega)} \le M$ 
on $[0,T)$, and thus,
$\| u \|_{L^\infty(\Omega)} \le M +1$ on $[0,T)$.
This and Lemmas \ref{full:Linf:v}, \ref{full:Linf:nabv} 
yield
\begin{align*}
  \lp{q}{(u\nabla v)(\cdot,s)} \le 
      (M+1) 
      \left( 
        \Kth \| \nabla v_0 \|_{L^q(\Omega)} e^{-\lambda s}
           + C_1 \| v_0 + \cmu w_0 \|_{L^\infty(\Omega)}  e^{-\gamma s} 
      \right)
\end{align*}
and 
\begin{align*}
  \lp{\infty}{(uv)(\cdot,s)} 
  \le (M+1) 
       e^{(1+\gamma)\ctime}
       \| v_0 + \cmu w_0 \|_{L^\infty(\Omega)}   
       e^{-\gamma s}, 
\end{align*}
which entails
  \begin{align*}
    &\| u(\cdot,t)-\beta \|_{L^\infty(\Omega)}
    \\
    &\le \alpha e^{-t}
           + \Kth\Kf (M+1) \| \nabla v_0 \|_{L^q(\Omega)} 
              \int^t_0 (1+(t-s)^{-\frac12-\frac{n}{2q}})e^{-(\lambda+1)(t-s)}e^{-\lambda s} \,ds \\
    &\quad\,
           + \cCwt \Kf (M+1) \| v_0 + \cmu w_0 \|_{L^\infty(\Omega)} 
              \int^t_0 (1+(t-s)^{-\frac12-\frac{n}{2q}})e^{-(\lambda+1)(t-s)}e^{-\gamma s} \,ds \\
    &\quad\,
           + (M+1) e^{(1+\gamma)\ctime} \| v_0 + \cmu w_0 \|_{L^\infty(\Omega)}
              \int^t_0 e^{-(t-s)} e^{-\gamma s}\,ds
  \end{align*}
for all $t\in[0,T)$. 
Here, noting that
$\int^t_0 e^{-(t-s)} e^{-\gamma s}\,ds
    \le \frac1{1-\gamma} e^{-\gamma t}$
and
  \[
    \int^t_0 (1+(t-s)^{-\frac12-\frac{n}{2q}})e^{-(\lambda+1)(t-s)}e^{-\lambda s}\,ds
    \le S_2 e^{-\lambda t}
  \]
as well as 
  \[
    \int^t_0 (1+(t-s)^{-\frac12-\frac{n}{2q}})e^{-(\lambda+1)(t-s)}e^{-\gamma s}\,ds
    \le S_3 e^{-\gamma t},
  \]
where 
$S_2=\int^\infty_0 (1+\sigma^{-\frac12-\frac{n}{2q}})e^{-\sigma} \,d\sigma < \infty$
and
$S_3=\int^\infty_0 (1+\sigma^{-\frac12-\frac{n}{2q}})e^{-(\lambda+1-\gamma) \sigma} \,d\sigma < \infty$,
we obtain
  \begin{align*}
    \| u(\cdot,t)-\beta \|_{L^\infty(\Omega)}
    \le \alpha e^{-t}
           + \Kth \Kf S_2 (M+1) \| \nabla v_0 \|_{L^q(\Omega)} 
              e^{-\lambda t} 
           + \cChat \|v_0+\xi w_0\|_{L^\infty(\Omega)} e^{-\gamma t}
  \end{align*}
for all $t\in[0,T)$, 
where
\begin{align*}
    \cChat 
    &:= C_1 \Kf S_3 (M+1) + \frac{(M+1) e^{(1+\gamma)\ctime}}{1-\gamma}
    \\
    & = \left[
            \Ktw \Kf S_1 S_3 |\Omega|^{\frac 1q} 
            \left(M+\frac{\mu}{\xi}\right) 
            + \frac{1}{1-\gamma} 
          \right] 
          (M+1)e^{(1+\gamma)\ctime}.
\end{align*}
Thus we can arrive at the conclusion.
\end{proof}

To see $\lp{\infty}{u(\cdot,T)-\beta} < g(T)$, 
we then show the following lemma. 

\begin{lem}\label{full:lem;KeyLem}
Let $\ep \in (0,1]$ and 
let $\xi, \delta, \gamma>0$ satisfy \eqref{dxg}. 
Then if $(u_0,v_0,w_0,z_0)$ fulfills \eqref{initial} and 
\begin{align}\label{initial:small}
\lp{\infty}{v_0+\xi w_0}\le \ep 
\qquad  \mbox{and} \qquad 
\lp{q}{\nabla v_0} \le \sqrt{\ep}, 
\end{align}
then 
  \[
    \| u(\cdot,t)-\beta \|_{L^\infty(\Omega)}
    \le \alpha e^{-t}
         + \cKlam \| \nabla v_0 \|_{L^q(\Omega)}e^{-\lambda t}
         + D \ep e^{-\gamma t}
    \quad\mbox{for all $t\in [0,T)$},
  \]
where $D=D(\alpha,\beta,\eta,\xi,\delta,\gamma, q,\Omega)$ is a positive constant defined as
\begin{align*}
  D(\alpha,\beta,\eta,\xi,\delta,\gamma, q,\Omega) 
  &:= \Kth \Kf S_2 \cKlam 
  +  \left[
       \Ktw \Kf S_1 S_3 |\Omega|^{\frac 1q} 
       \wt{M}
       + \frac{1}{1-\gamma} 
    \right]  \wt{M} e^{(1+\gamma) \ctime}, 
  \\
  \wt{M} &:= (1 + \Kth\Kf S_2 )(\alpha + \beta + \eta).
\end{align*}
\end{lem}
\begin{proof} 
In view of the fact that $\lp{q}{\nabla v_0} \le \sqrt{\ep} \le1$, 
Lemma \ref{full:lem;esti;u-beta} 
enables us to obtain
  \[
    \| u(\cdot,t) - \beta \|_{L^\infty(\Omega)}
    \le \alpha e^{-t}
           + \Kth\Kf S_2 (M+1) \| \nabla v_0 \|_{L^q(\Omega)} 
              e^{-\lambda t} 
           + \cChat \|v_0 + \xi w_0 \|_{L^\infty(\Omega)} e^{-\gamma t}
 \]
for all $t\in[0,T)$.
Noting from \eqref{cKlam} that 
$\cKlam = \Kth\Kf S_2 (\alpha + \beta + \eta)$ holds and recalling the definition of $M$ in Lemma \ref{full:Linf:nabv} and
the second condition in \eqref{initial:small},
we see that 
\begin{align*}
  \Kth\Kf S_2 (M+1) \| \nabla v_0 \|_{L^q(\Omega)} 
  &= \Kth\Kf S_2 (\alpha + \beta + \eta + \cKlam  \| \nabla v_0 \|_{L^q(\Omega)} ) 
      \| \nabla v_0 \|_{L^q(\Omega)}  
\\
  &= \Kth\Kf S_2 (\alpha + \beta + \eta ) \| \nabla v_0 \|_{L^q(\Omega)} 
       + \Kth\Kf S_2 \cKlam \| \nabla v_0 \|_{L^q(\Omega)}^2
\\
  &\le \cKlam \| \nabla v_0 \|_{L^q(\Omega)} 
       + \Kth\Kf S_2 \cKlam \ep.
\end{align*}
Moreover, since the fact that $\frac{\mu}{\xi}<1$ (from \eqref{dxg}) entails 
\begin{align*}
 M+\frac\mu\xi< M + 1
 \le \alpha + \beta + \eta + \Kth\Kf S_2 (\alpha + \beta + \eta)
 = (1 + \Kth\Kf S_2 )(\alpha + \beta + \eta) 
 = \wt{M}, 
\end{align*}
we have from the definition of $C_2$ in Lemma \ref{full:lem;esti;u-beta} and the first condition in \eqref{initial:small} 
that 
\begin{align*}
 C_2 \lp{\infty}{v_0 + \xi w_0} \le 
          \left[
            K_2 S_1 S_3 \Kf |\Omega|^{\frac 1q} 
            \wt{M}
            + \frac{1}{1-\gamma} 
          \right] \wt{M} e^{(1+\gamma) \ctime} \ep. 
\end{align*}
Thus, combining the above inequalities,
we attain that 
\begin{align*}
   \| u(\cdot,t)-\beta \|_{L^\infty(\Omega)}
   &\le \alpha e^{-t}
         + \cKlam \| \nabla v_0 \|_{L^q(\Omega)}e^{-\lambda t} 
         + \Kth\Kf S_2 \cKlam \ep \cdot e^{-\gamma t}         
         \\
         &\quad\,
  + \left[
       \Ktw \Kf S_1S_3 |\Omega|^{\frac 1q} 
       \wt{M}
       + \frac{1}{1-\gamma} 
    \right]  \wt{M} e^{(1+\gamma) \ctime} \ep \cdot e^{-\gamma t}
\end{align*}
for all $t\in[0,T)$,
which implies the conclusion. 
\end{proof}

By choosing $\varepsilon>0$ satisfying $D\varepsilon < \eta$, we can show that $T=\tmax$ holds.  


\begin{lem}\label{beh:u}
Let $\xi, \delta, \gamma>0$ satisfy \eqref{dxg}. 
Then there is $\varepsilon_1 = \varepsilon_1(\alpha,\beta,\eta,\xi,\delta,\gamma, q,\Omega)  > 0$
satisfying the following property\/{\rm :} 
If $(u_0,v_0,w_0,z_0)$ fulfills \eqref{initial} and \eqref{initial:small} with $\ep \le \ep_1$, 
then $T=\tmax$ and
  \[
    \| u(\cdot,t)-\beta \|_{L^\infty(\Omega)}
    \le \alpha e^{-t}
         + \cKlam \| \nabla v_0 \|_{L^q(\Omega)}e^{-\lambda t}
         + \eta e^{-\gamma t}
  \]
as well as 
  \[
    \| v (\cdot,t) + \cmu w (\cdot,t) \|_{L^\infty(\Omega)} 
    \le \|v_0+\xi w_0\|_{L^\infty(\Omega)}
         e^{(1+\gamma)\ctime}  
         e^{-\gamma t}, 
         \quad \lp{q}{\nabla v(\cdot,t)}\le Ce^{-\gamma t}
  \]
for all $t\in [0,\tmax)$ with some $C>0$.  
\end{lem}
\begin{proof}
To show that $T=\tmax$ by a contradiction argument,
we assume that $T<\tmax$. 
Now we take $\ep_1 = \min\{1,\frac\eta{2D}\}$, 
where
$D=D(\alpha,\beta,\eta,\xi,\delta,\gamma, q,\Omega)$ 
is a constant defined in Lemma \ref{full:lem;KeyLem}, and
assume \eqref{initial:small} with $\ep \le \ep_1$.
By virtue of Lemma \ref{full:lem;KeyLem} 
we have
\begin{align*}
   \| u(\cdot,t)-\beta \|_{L^\infty(\Omega)}
    \le \alpha e^{-t}
         + \cKlam \| \nabla v_0 \|_{L^q(\Omega)}e^{-\lambda t}
         + \frac\eta2 e^{-\gamma t}
   \quad\mbox{for all $t\in [0,T)$},
\end{align*}
and then this contradicts with 
the definition of $T$ in \eqref{T}. 
Consequently, we see that $T=\tmax$, 
and the definition of $T$ and Lemmas \ref{full:Linf:v}, \ref{full:Linf:nabv} lead to the conclusion.
\end{proof}

\section{Decay estimates for {\boldmath $\nabla w$} and {\boldmath $z$}: Proof of Theorem \ref{thm}}\label{Sec;BddOfZ}

In this section we shall show decay estimates for $\nabla w,z$, and then attain Theorem \ref{thm}. 
We first prepare existence of a useful test function $\varphi$ which is used in the proof of \cite[Lemma 6.2]{TW2012_consumption}.


\begin{lem}\label{phi}
For all $p>1$ and $\ell \in (0,p-1)$ 
there exists $b_0>0$ such that 
for every $\wmax\in(0,b_0)$ 
one can find $\kappa=\kappa(b)>0$ with the property that
the function $\varphi$ defined as
$\varphi(y):=(2\wmax-y)^{-\ell}$
for $y\in[0,\wmax]$
satisfies that 
  \[
    \varphi''(y)-p\varphi'(y)>0
    \quad\mbox{and}\quad
    p(p-1)\varphi(y) 
      - \frac{\big[-2p \varphi'(y) + p(p-1) \varphi(y) \big]^2}{4\big[ \varphi''(y) - p \varphi'(y) \big]}
\ge \kappa
  \]
for all $y\in(0,\wmax)$.
\end{lem}

\begin{proof}
It follows from the relation $\ell<p-1$ that 
$p-1 - \frac{p\ell^2}{\ell(\ell+1)} 
= \frac{(p-1)(\ell+1)-p\ell}{\ell+1}
= \frac{p-\ell-1}{\ell+1}>0$. 
Therefore we can find $b_0>0$ small enough
such that for every $\wmax \in (0,b_0)$,
  \begin{align}\label{wmax}
    \wmax < \frac{\ell+1}{2p} 
    \quad\mbox{and}\quad
    p-1 - \frac{p\ell^2+p(p-1)^2\wmax^{2}}{\ell( \ell+1 - 2p\wmax )}>0.
  \end{align}
Also we observe that 
  \[
    \varphi'(y)=\ell(2\wmax-y)^{-\ell-1}
    \quad\mbox{and}\quad
    \varphi''(y)=\ell(\ell+1)(2\wmax-y)^{-\ell-2}
    \quad\mbox{for all $y\in(0,\wmax)$}.
  \]
By virtue of the first relation in \eqref{wmax} 
we see that
  \begin{align*}
    \varphi''(y)-p\varphi'(y)
    &=( \ell+1 - p(2\wmax-y) ) \cdot \ell(2\wmax-y)^{-\ell-2}
    \\
    &\ge( \ell+1 - 2p\wmax ) \cdot \ell(2\wmax)^{-\ell-2}
    >0
    \quad\mbox{for all $y\in(0,\wmax)$}.
  \end{align*}
Moreover, from the straightforward calculations 
we have 
  \begin{align*}
    &p(p-1)\varphi(y) 
      - \frac{\big[-2p \varphi'(y) + p(p-1) \varphi(y) \big]^2}{4\big[ \varphi''(y) - p \varphi'(y) \big]}
    \\
    &\ge
        p(p-1)\varphi(y) 
        - \frac{4p^2\varphi'^2(y) 
                    +p^2(p-1)^2\varphi^2(y)}
                  {4\big[ \varphi''(y) - p \varphi'(y) \big]}
    \\
    &=
        p(p-1)(2\wmax-y)^{-\ell}
        - \frac{4p^2\ell^2(2\wmax-y)^{-2\ell-2} 
                     +p^2(p-1)^2(2\wmax-y)^{-2\ell}}
                  {4( \ell+1 - p(2\wmax-y) ) \cdot \ell(2\wmax-y)^{-\ell-2} }
    \\
    &=p(2\wmax-y)^{-\ell}
        \left[
        p-1
        - \frac{4p\ell^2 
                     +p(p-1)^2(2\wmax-y)^{2}}
                  {4\ell( \ell+1 - p(2\wmax-y) )}
        \right]
    \\
    &\ge
        p(2\wmax)^{-\ell}
        \left[
        p-1
        - \frac{p\ell^2 
                     +p(p-1)^2\wmax^{2}}
                  {\ell( \ell+1 - 2p\wmax )}
        \right] 
        =:\kappa
  \end{align*}
for all $y\in(0,\wmax)$.
The second relation in \eqref{wmax} ensures that $\kappa>0$.
\end{proof}

In the following 
we let $\xi, \delta, \gamma>0$ satisfy \eqref{dxg} and let $q >n$, as well as fix  
\[ p > \frac{qn}{q-n} (\ge n).\]  
We then show an $L^p$-estimate for $z$ by referring the proof of \cite[Lemma 6.2]{TW2012_consumption}. 
\begin{lem}\label{Lpz}
There exists $\zeta>0$ 
with the following property\/{\rm :} 
If $(u_0,v_0,w_0,z_0)$ fulfills \eqref{initial} and \eqref{initial:small} with $\ep \le \ep_1$, where $\ep_1$ is given by Lemma \ref{beh:u}, 
and if there exists $t_0\in[0,\tmax)$ such that 
  \begin{align}\label{assum:small}
    \| v(\cdot,t) + \xi w(\cdot,t) \|_{L^\infty(\Omega)} \le \zeta
    \quad\mbox{for all $t\in[t_0,\tmax)$},
  \end{align}
then there exists $C>0$ such that 
  \begin{align}\label{lpesti}
    \|z(\cdot,t)\|_{L^p(\Omega)} \le C
    \quad\mbox{for all $t\in[0,\tmax)$}.
  \end{align}
\end{lem}

\begin{proof}
Lemma \ref{beh:u} implies 
  \[
    \| u(\cdot,t) \|_{L^\infty(\Omega)}
    \le 
    \alpha + \cKlam \| \nabla v_0 \|_{L^q(\Omega)} + \eta + \beta =:\umax
    \quad\mbox{on $[0,\tmax)$}.
  \]
Let  
$\ell \in (0, p-1)$,
and choose 
$\wmax \in(0,b_0)$, 
where $b_0$ is given by Lemma \ref{phi}. 
We also assume that there exists $t_0 \in [0,\tmax)$ fulfilling
\eqref{assum:small} with $\zeta:=\xi\wmax$, 
which ensures that 
  \[
    0 \le w \le \wmax
    \quad\mbox{in $\overline{\Omega}\times[t_0, \tmax)$}.
  \]
Multiplying the fourth equation in \eqref{P} by $z^{p-1}\varphi(w)$, 
where $\varphi$ is defined in Lemma \ref{phi},  
and integrating by parts, we have
  \begin{align}\label{ene}
    \frac{d}{dt} \int_\Omega z^p \varphi(w)
    &= -p(p-1)\int_\Omega z^{p-2}\varphi(w) |\nabla z|^2
         -p\int_\Omega z^{p-1} \varphi'(w) \nabla w \cdot \nabla z
         \notag\\ &\quad\,
         + p(p-1)\int_\Omega z^{p-1} \varphi(w) \nabla w \cdot \nabla z
         + p\int_\Omega z^p \varphi'(w) |\nabla w|^2
         \notag\\ &\quad\,
         + p\int_\Omega z^p f(w) \varphi(w)
         - p\int_\Omega z^p \varphi(w) 
    \notag\\ &\quad\,
         - p\int_\Omega z^{p-1} \varphi'(w) \nabla w \cdot \nabla z
         - \int_\Omega z^p \varphi''(w) |\nabla w|^2
    \notag\\ &\quad\,
         + \int_\Omega uvz^p \varphi'(w)
         - \int_\Omega z^{p+1}w \varphi'(w)
         - \int_\Omega z^pw \varphi'(w)
    \notag\\
    &= -p(p-1)\int_\Omega z^{p-2}\varphi(w) |\nabla z|^2
         \notag\\ &\quad\,     
         + \int_\Omega z^{p-1} \big[-2p \varphi'(w) + p(p-1) \varphi(w) \big] \nabla w \cdot \nabla z
         \notag\\ &\quad\,
         - \int_\Omega z^p \big[ \varphi''(w) - p \varphi'(w) \big] |\nabla w|^2
         \notag\\ &\quad\,
         + p\int_\Omega z^p f(w) \varphi(w)
         - p\int_\Omega z^p \varphi(w) 
    \notag\\ &\quad\,
         + \int_\Omega uvz^p \varphi'(w)
         - \int_\Omega z^{p+1}w \varphi'(w)
         - \int_\Omega z^pw \varphi'(w)
  \end{align}
on $(t_0,\tmax)$. 
Noting from Lemma \ref{phi} that $\varphi''(w) - p \varphi'(w)>0$ 
on $(t_0,\tmax)$, 
we can utilize Young's inequality to obtain 
  \begin{align*}
    &z^{p-1} \big[-2p \varphi'(w) + p(p-1) \varphi(w) \big] \nabla w \cdot \nabla z
    \\
    &\le z^p \big[ \varphi''(w) - p \varphi'(w) \big] |\nabla w|^2
           + z^{p-2} \frac{\big[-2p \varphi'(w) + p(p-1) \varphi(w) \big]^2}{4\big[ \varphi''(w) - p \varphi'(w) \big]} |\nabla z|^2
  \end{align*}
on $(t_0,\tmax)$. Moreover, Lemma \ref{phi} guarantees that 
  \[
    -p(p-1)\varphi(w)
    +\frac{\big[-2p \varphi'(w) + p(p-1) \varphi(w) \big]^2}{4\big[ \varphi''(w) - p \varphi'(w) \big]}
    \le -\kappa
    \quad\mbox{on $(t_0,\tmax)$}.
  \]
Thus, the above inequalities together with \eqref{ene} and $\varphi'>0$ yield 
  \[
    \frac{d}{dt} \int_\Omega z^p \varphi(w)
    + p\int_\Omega z^p \varphi(w)
    \le - \kappa \int_\Omega z^{p-2}|\nabla z|^2
         + p\int_\Omega z^p f(w) \varphi(w) 
         + \int_\Omega uvz^p \varphi'(w)
  \]
on $(t_0,\tmax)$. 
Since $\varphi(w) \le \wmax^{-\ell}$ and $\varphi'(w) \le \ell \wmax^{-\ell-1}$ 
as well as 
$uv \le c_1$ with some $c_1>0$ by Lemma \ref{beh:u}, 
there exists $c_2>0$ such that 
  \[
    p\int_\Omega z^p f(w) \varphi(w)
    + \int_\Omega uvz^p \varphi'(w)
    \le 
    c_2 \int_\Omega z^p
    \quad\mbox{on $(t_0,\tmax)$}.
  \]
Making use of the Gagliardo--Nirenberg inequality 
and Lemma \ref{mass} as well as 
Young's inequality, 
we obtain constants $c_3, c_4>0$ such that 
  \begin{align*}
    \int_\Omega z^p  
    \le 
    \| z^\frac{p}2 \|_{L^2(\Omega)}^2
    &\le 
    c_3
    \left(
    \| \nabla z^\frac{p}2 \|_{L^2(\Omega)}^{2\theta}
    \| z^\frac{p}2 \|_{L^\frac2p(\Omega)}^{2(1-\theta)}
    +
    \| z^\frac{p}2 \|_{L^\frac2p(\Omega)}^2
    \right)
    \\
    &\le 
    \frac{4\kappa}{c_2p^2} \| \nabla z^\frac{p}2 \|_{L^2(\Omega)}^{2} + c_4
    \quad\mbox{on $(t_0,\tmax)$},
  \end{align*}
where 
$\theta = \frac{\frac{p}2-\frac12}{\frac{p}2+\frac1n-\frac12} \in (0,1)$.
Hence we have
  \[
    \frac{d}{dt} \int_\Omega z^p \varphi(w)
    + p\int_\Omega z^p \varphi(w)
    \le c_2c_4
    \quad\mbox{on $(t_0,\tmax)$}.
  \]
This implies that
  \[
    \int_\Omega z^p(\cdot,t) \varphi(w(\cdot,t)) 
    \le \int_\Omega z^p(\cdot,t_0) \varphi(w(\cdot,t_0)) 
         + \frac{c_2c_4}p 
    \quad\mbox{for all $t\in(t_0,\tmax)$}, 
  \]
which together with the relation $\varphi(w) \ge (2b)^{-\ell}$ 
leads to
  \[
    \int_\Omega z^p(\cdot,t) 
    \le (2b)^{-\ell} 
         \left(
           \int_\Omega z^p(\cdot,t_0) \varphi(w(\cdot,t_0)) 
           + \frac{c_2c_4}p 
         \right)
         =:c_5
    \quad\mbox{for all $t\in(t_0,\tmax)$}.
  \]
Hence we arrive at the conclusion with 
$C:=\max\{ c_5, \sup_{t\in[0,t_0]} \| z(\cdot,t) \|_{L^\infty(\Omega)} \}$.
\end{proof}


\begin{corollary}\label{cor}
Let 
  \[
  \ep_2
  :=
  \begin{cases}
  \ep_1 &\mbox{if $n=2$},\\
  \min\{\ep_1,\zeta e^{-(1+\gamma)\ctime}\} &\mbox{if $n\ge3$},
  \end{cases}
  \]
where $\ep_1$ and $\zeta$ are defined 
in Lemmas \ref{beh:u} and \ref{Lpz}, respectively.
If $(u_0,v_0,w_0,z_0)$ satisfies \eqref{initial} and \eqref{initial:small} with $\ep \le \ep_2$, 
then \eqref{lpesti} holds with some $C>0$.
\end{corollary}

\begin{proof}
It is sufficient to confirm that \eqref{assum:small} holds with some $t_0\in[0,\tmax)$.
In the case $n \ge 3$, 
combining Lemma \ref{beh:u} and \eqref{initial:small} with $\ep \le \ep_2=\min\{\ep_1,\zeta e^{-(1+\gamma)\ctime}\}$ enables us to obtain \eqref{assum:small} with $t_0=0$. 
On the other hand, in the case $n=2$, 
we see from Corollary \ref{cor2d} that 
\eqref{P} admits 
a corresponding global classical solution 
in \eqref{P} for every initial data.
Thus for the initial data satisfying \eqref{initial:small} with $\ep\le\ep_2=\ep_1$, 
Lemma \ref{beh:u} is established with $\tmax=\infty$. 
Therefore we can take $t_0\in[0,\infty)$ 
fulfilling \eqref{assum:small}.
Hence this corollary results from Lemma \ref{Lpz}.
\end{proof}

Thanks to 
Lemma \ref{beh:u} and 
Corollary \ref{cor}, we can show an $L^q$-decay estimate for $\nabla w$ via semigroup estimates. 

\begin{lem}\label{esti:nabwz}
If $(u_0,v_0,w_0,z_0)$ fulfills
\eqref{initial} and \eqref{initial:small}
with $\ep \le \ep_2$, which is defined in Corollary \ref{cor}, 
then there is $C>0$ such that 
  \[
    \| \nabla w(\cdot,t) \|_{L^q(\Omega)}
    \le C e^{-\gamma t}
    \quad\mbox{for all $t\in[0,\tmax)$}.
  \]
\end{lem}

\begin{proof}
From the third equation in \eqref{P} and 
semigroup estimates in Lemma \ref{semi}, 
  \begin{align*}
    &\| \nabla w(\cdot,t) \|_{L^q(\Omega)}
    \\
    &\le 
    \| \nabla e^{t(\Delta-1)} w_0 \|_{L^q(\Omega)}
    + \int_0^t \| \nabla e^{(t-s)(\Delta-1)} [uv-wz] \|_{L^q(\Omega)} \,ds
    \\
    &\le 
    \Kth e^{-(\lambda+1) t} \| \nabla w_0 \|_{L^q(\Omega)}
    + \Ktw |\Omega|^{\frac 1q} \int_0^t (1+(t-s)^{-\frac12})e^{-(\lambda+1)(t-s)} \| uv \|_{L^\infty(\Omega)} \,ds
    \\
    &\quad\,
    + \Ktw \int_0^t (1+(t-s)^{-\frac12 - \frac{n}{2}\left(\frac1p-\frac1q\right)})e^{-(\lambda+1)(t-s)} \| wz \|_{L^p(\Omega)} \,ds
  \end{align*}
for all $t\in[0,\tmax)$.
Noting from the fact $p >\frac{qn}{q-n}>\frac{qn}{q+n}$ that $-\frac12 - \frac{n}{2}\big(\frac1p-\frac1q\big)>-1$,  
thanks to Lemma \ref{beh:u} and Corollary \ref{cor} 
we have 
  \[
    \| (uv)(\cdot,s) \|_{L^\infty(\Omega)}
    \le 
    c_1e^{-\gamma s}
\quad\mbox{and}\quad 
    \| (wz)(\cdot,s) \|_{L^p(\Omega)}
    \le 
    c_2e^{-\gamma s}
  \]
for all $s\in(0,\tmax)$ with some $c_1,c_2>0$, 
and hence
a combination of the above inequalities 
leads to the conclusion.
\end{proof}

We finally attain an $L^\infty$-decay estimate for $z$. 

\begin{lem}\label{bddz}
If $(u_0,v_0,w_0,z_0)$ satisfies 
\eqref{initial} and \eqref{initial:small}
with $\ep \le \ep_2$, which is defined in Corollary \ref{cor}, 
then there is $C>0$ such that
  \[
    \|z(\cdot,t)\|_{L^\infty(\Omega)}
    \le 
    C e^{-\gamma t}
    \quad\mbox{for all $t\in[0,\tmax)$}.
  \]
\end{lem}

\begin{proof}
We let $r=\frac{pq}{p+q}$. Then we have $r > n$ (from $p > \frac{qn}{q-n}$) and $r < q$ as well as $p=\frac{qr}{q-r}$. 
From the fourth equation in \eqref{P}
and semigroup estimates in Lemma \ref{semi} 
we see that
  \begin{align}\label{zlinf}
    \| z(\cdot,t) \|_{L^\infty(\Omega)}
    &\le \| e^{t(\Delta-1)}z_0 \|_{L^\infty(\Omega)}
         + \int_0^t \| e^{(t-s)(\Delta-1)} \nabla \cdot (z \nabla w) \|_{L^\infty(\Omega)} \,ds
    \notag \\
    &\quad\,
         + \int_0^t \| e^{(t-s)(\Delta-1)} wz \|_{L^\infty(\Omega)} \,ds
    \notag \\
    &\le e^{-t} \| z_0 \|_{L^\infty(\Omega)} 
         + \Kf \int_0^t (1+(t-s)^{-\frac12-\frac{n}{2r}})e^{-(\lambda+1)(t-s)}\| z \nabla w \|_{L^r(\Omega)} \,ds
    \notag \\
    &\quad\,
         + \Ko \int_0^t (1+(t-s)^{-\frac{n}{2p}})e^{-(\lambda+1)(t-s)}\| wz \|_{L^p(\Omega)} \,ds
  \end{align}
for all $t\in[0,\tmax)$.
By Lemmas \ref{beh:u}, \ref{esti:nabwz} and Corollary \ref{cor}, 
we obtain constants $c_1,c_2>0$ such that
  \[
    \| (z \nabla w)(\cdot,s) \|_{L^r(\Omega)} \le \lp{p}{z(\cdot,s)}\lp{q}{\nabla w(\cdot,s)}
    \le 
    c_1 e^{-\gamma s}
  \]
and 
 \[
   \| (wz)(\cdot,s) \|_{L^p(\Omega)}
   \le 
   c_2 e^{-\gamma s}
 \]
for all $s \in (0,\tmax)$. 
Inserting these inequalities into \eqref{zlinf} yields the conclusion.
\end{proof}

\begin{proof}[Proof of Theorem \ref{thm}]
Lemmas \ref{beh:u}, \ref{esti:nabwz} and \ref{bddz} together with \eqref{criterion}
yield $\tmax=\infty$ and the conclusion of the theorem.
\end{proof}

\section*{Declaration of competing interest}

The authors declare that they have no known competing financial interests or personal relationships that could have appeared
to influence the work reported in this paper.

\section*{Acknowledgment}
The authors M.M. and Y.T. are partially supported by JSPS KAKENHI Grant Number JP25K17265 and JP24K22844, respectively.

\bibliographystyle{plain}
\bibliography{paper_MT_025_revision_arXiv.bbl}
\end{document}